\documentclass[11pt,reqno]{amsart}
\usepackage{tikz}
\usepackage{subfig}
\usepackage{epstopdf}
\usepackage{epsfig}
\usepackage{math}
\graphicspath{{./figures/}}
\usepackage{tikz}
\usetikzlibrary{shapes,arrows}
\tikzstyle{decision} = [diamond, draw, fill=blue!20, 
    text width=4.5em, text badly centered, node distance=3cm, inner sep=0pt]
\tikzstyle{block} = [rectangle, draw, fill=blue!20, 
    text width=5em, text centered, rounded corners, minimum height=4em]
\tikzstyle{line} = [draw, -latex']
\tikzstyle{cloud} = [draw, ellipse,fill=red!20, node distance=3cm,
    minimum height=2em]

\newcommand{\rt}{\mathrm{T}}

\usetikzlibrary{positioning}
\tikzset{main node/.style={circle,fill=blue!20,draw,minimum size=1cm,inner sep=0pt},  }

\begin{document}
\title[Transport $f$-divergences]{Transport $f$-Divergences}
\author[Li]{Wuchen Li}
\email{wuchen@mailbox.sc.edu}
\address{Department of Mathematics, University of South Carolina, Columbia.}
\begin{abstract}
We define a class of divergences to measure differences between probability density functions in one-dimensional sample space. The construction is based on the convex function with the Jacobi operator of mapping function that pushforwards one density to the other. We call these information measures {\em transport $f$-divergences}. We present several properties of transport $f$-divergences, including invariances, convexities, variational formulations, and Taylor expansions in terms of mapping functions. Examples of transport $f$-divergences in generative models are provided. 
\end{abstract}
\maketitle

\section{Introduction} 
Measuring dissimilarities between probability densities are crucial problems in machine learning \cite{CA} and Bayesian inference problems \cite{GAN}. The information theory studies these dissimilarity functionals \cite{IG2, CoverThomas1991_elements}. In this area, $f$-divergences invented by Csiszar-Morimoto \cite{CS} and Ali-Silvey \cite{AS}, belong to a class of information measurements. Famous examples of $f$-divergences include total variation (TV) distance, $\chi^2$-divergence, Kullback-Leibler (KL) divergence, Jensen-Shannon (JS) divergence \cite{SS} and $\alpha$-divergences \cite{IG1,CA}. They have been widely applied in image and signal processing \cite{Nielsen}, Markov Chain Monte Carlo (MCMC) sampling algorithms \cite{LiHess1}, Bayesian inverse problems, and generative modeling in artificial intelligence (Generative AI) \cite{GAN} .    

In recent years, optimal transport \cite{Villani2009_optimal} studies the other types of distances between probability densities, which has shown popularities in image or signal processing \cite{Nielsen} and generative AI \cite{ACB}. In this area, the distance function, called earth mover's distance or Wasserstein distance, also measures the differences between probability densities by comparing pushforward mapping functions between densities from a ground cost in the sample space. It is known that the mapping function can measure the probability densities with sparse support, such as empirical data distributions, while classical information divergences may not be well-defined \cite{ACB, GAN}. In particular, optimal transport studies a Riemannian metric in probability densities space, namely Wasserstein-$2$ space \cite{AGS, Villani2009_optimal} or density manifold \cite{Lafferty}. Recently, the convexities of information entropy functionals in Wasserstein-$2$ space are used in information-theoretical inequalities \cite{AGS,ST}.  

The asymmetry between probability densities is a property in studying information measures and their variational problems. In $f$-divergence, the asymmetry property is constructed from the ratio between two probability density functions, called likelihood functions. However, the asymmetry property is often lacking in classical optimal transport distances, especially with Euclidean distances as ground costs. A natural question arises. {\em What are analogs of asymmetry property and $f$-divergences in Wasserstein-$2$ space?} 

This paper defines a class of divergences in one-dimensional probability space. Let $p$, $q$ be probability densities supported on $\Omega=\mathbb{R}^1$. Consider 
\begin{equation*}
\mathrm{D}_{\mathrm{T}, f}(p\|q)=\int_\Omega f(\frac{q(x)}{p(T(x))})q(x)dx,
\end{equation*}
where $f\colon \mathbb{R} \rightarrow \mathbb{R}_+$ is a positive convex function with $f(1)=0$, and $T\colon \Omega \rightarrow\Omega$ is a monotone function, which pushforwards densities $q$ to $p$. We call $\mathrm{D}_{\mathrm{T}, f}$ the {\em transport $f$-divergence.} We apply the ratio between $q$ and $p(T)$ to represent the likelihood function. Several properties of transport $f$-divergences are studied, including invariances, convexities, variational formulations, local behaviors, and Taylor expansions in Wasserstein-$2$ space. We also provide several examples of transport $f$-divergences and their formulas in generative models. 

In literature, there are joint studies between information divergences and optimal transport distances \cite{BB,RGAN, KM,ST}. On the one hand, \cite{ST} applies the second-order derivatives of information divergences in Wasserstein-2 space to prove the first-order entropy power inequalities and their generalizations. On the other hand, the analytical and statistical estimation properties of Wasserstein-$2$ distances have been conducted in Gaussian distributions \cite{KM}. Compared to previous works, we define transport $f$-divergences, which generalizes KL divergences, Hessian distances, and $\alpha$-divergences in Wasserstein-$2$ space \cite{LiHess1, LiD, LiHess2}. 

This paper is organized as follows. In section \ref{sec2}, we briefly review $f$-divergences and their properties. In section \ref{sec3}, we introduce the main result of this paper. We define transport $f$-divergences and formulate their properties. We present several examples of transport $f$-divergences in section \ref{sec4}. Several analytical formulas of transport $f$-divergences in location-scale families and generative models are provided in section \ref{sec5}.

\section{Review of $f$-divergences}\label{sec2}
 In this section, we briefly review $f$-divergences \cite{AS, CS}, which measures the difference between probability distributions.  
 
 Consider a one-dimensional sample space $\Omega=\mathbb{R}^1$. Denote the space of smooth positive probability density functions by
\begin{equation*}
\mathcal{P}(\Omega)=\Big\{p\in C^{\infty}(\Omega)\colon \int_\Omega p(x)dx=1,~p(x)> 0\Big\}.
\end{equation*}
Given two probability density functions $p, q\in\mathcal{P}(\Omega)$, define $f$-divergence $\mathrm{D}_f\colon\mathcal{P}(\Omega)\times\mathcal{P}(\Omega)\rightarrow\mathbb{R}_+$ by
\begin{equation*}
\mathrm{D}_f(p\|q)=\int_\Omega f(\frac{p(x)}{q(x)})q(x)dx,
\end{equation*}
where $f\colon \mathbb{R}\rightarrow\mathbb{R}_+$ is a convex function with $f(1)=0$. In general, $\mathrm{D}_f(p\|q)$ is not symmetric with respect to densities $p$, $q$. I.e., $\mathrm{D}_{f}(p\|q)\neq \mathrm{D}_f(q\|p)$. From this reason, we call $\mathrm{D}_f$ the divergence function, instead of the distance function. 

There are several examples of $f$-divergences. 
\begin{itemize}
\item[(1)] Total variation: If $f(u)=|u-1|$, then 
\begin{equation*}
\mathrm{D}_f(p\|q)=\mathrm{D}_{\mathrm{TV}}(p,q)=\int_\Omega |p(x)-q(x)|dx.
\end{equation*}
\item[(2)] $\chi^2$-divergence: If $f(u)=|u-1|^2$, then 
\begin{equation*}
\mathrm{D}_f(p\|q)=\mathrm{D}_{\chi^2}(p\|q)=\int_\Omega \frac{|p(x)-q(x)|^2}{q(x)}dx.
\end{equation*}
\item[(3)] Squared Hellinger distance: If $f(u)=\frac{1}{2}|1-\sqrt{u}|^2$, then 
\begin{equation*}
\mathrm{D}_f(p\|q)=\frac{1}{2}\int_\Omega (\sqrt{p(x)}-\sqrt{q(x)})^2dx.
\end{equation*}
\item[(4)] KL divergence: If $f(u)=u\log u-(u-1)$, then
\begin{equation*}
\mathrm{D}_f(p\|q)=\mathrm{D}_{\mathrm{KL}}(p\|q)=\int_\Omega p(x)\log\frac{p(x)}{q(x)}dx.
\end{equation*}
\item[(5)] JS divergence: If $f(u)=\frac{1}{2}(u\log\frac{2u}{u+1}+\log\frac{2}{u+1})$, then 
\begin{equation*}
\mathrm{D}_{f}(p\|q)=\mathrm{D}_\mathrm{JS}(p,q)=\frac{1}{2}\Big(\mathrm{D}_{\mathrm{KL}}(p\|\frac{p+q}{2})+\mathrm{D}_{\mathrm{KL}}(q\|\frac{p+q}{2})\Big). 
\end{equation*}
\end{itemize}

The $f$-divergences exhibit several useful properties in estimations and AI sampling algorithms; see \cite{AC,SS}. 
\begin{itemize}
\item[(i)] Nonnegativity: The $f$-divergence is always nonnegative and equals to zero if and only if $p=q$. 
\item[(ii)] Generalized entropy: Let $q=1$. 
\begin{equation*}
\mathrm{D}_f(p\|1)=\int_\Omega f(p(x))dx. 
\end{equation*}
\item[(iii)] Joint convexity: $\mathrm{D}_f(p\|q)$ is jointly convex in both variables $p$ and $q$. Given any constant $\lambda\in [0,1]$, then 
\begin{equation*}
\mathrm{D}_f(\lambda p_1+(1-\lambda)p_2\|\lambda q_1+(1-\lambda)q_2)\leq \lambda \mathrm{D}_f(p_1\|q_1)+(1-\lambda)\mathrm{D}_f(p_2\|q_2).
\end{equation*}
\item[(iv)] Additivity and Scaling: Suppose that $f_1$, $f_2$ are convex functions and $a>0$, then 
\begin{equation*}
\mathrm{D}_{f_1+af_2}(p\|q)=\mathrm{D}_{f_1}(p\|q)+a\mathrm{D}_{f_2}(p\|q).
\end{equation*}
\item[(v)] Invariance: The $f$-divergence is invariant to bijective transformations. Suppose $k\colon \Omega\rightarrow\Omega$ is a bijective map function, then 
\begin{equation*}
\mathrm{D}_f(p\|q)=\mathrm{D}_f(k_\#p\|k_\#q).
\end{equation*}
Besides, denote $\tilde f(u)=f(\frac{1}{u})u$, then $\mathrm{D}_f(p\|q)=\mathrm{D}_{\tilde f}(q\|p)$.
\item[(vi)] Variational formulation: Assume $f$ is strictly convex, then
\begin{equation*}
\begin{split}
\mathrm{D}_f(p\|q)=&\sup_{\phi}~\int_\Omega\phi(x)p(x)dx-\int_\Omega f^*(\phi(x))q(x)dx,
\end{split}
\end{equation*}
where the infimum is taken among continuous function $\phi\in C^{1}(\Omega; \mathbb{R})$, and $f^*$ is the conjugate function of $f$, such that $f^*(v):=\sup_{v\in \mathbb{R}}~\{uv-f(u)\}$.

\item[(vii)] Local behaviors: Suppose $f\in C^{2}$, then 
\begin{equation*}
\lim_{\lambda\rightarrow 0}\frac{1}{\lambda^2}\mathrm{D}_f\big((1-\lambda)q+\lambda p\|q\big)=\frac{f''(1)}{2}\mathrm{D}_{\chi^2}(p\|q).
\end{equation*}
\item[(viii)] Taylor expansions: Suppose $f\in C^4$ and $f'(1)=0$, then  
\begin{equation*}
\begin{split}
\mathrm{D}_f(p\|q)=&\int_\Omega\Big[\frac{1}{2}f''(1) \frac{(p(x)-q(x))^2}{q(x)}+\frac{1}{6}f'''(1)\frac{(p(x)-q(x))^3}{q(x)^2}\Big]dx\\
&+O(\int_\Omega \frac{(p(x)-q(x))^4}{q(x)^{3}}dx).
\end{split}
\end{equation*}

\end{itemize}
\section{Transport $f$-divergences}\label{sec3}
In this section, we recall some facts in optimal transport. Using them, we formulate $f$-divergences in terms of optimal transport mapping functions. We call them {\em transport $f$-divergences}. We demonstrate several properties of transport $f$-divergences, including invariances, dualities, local behaviors and Taylor expansions in Wasserstein-2 spaces.   
 \subsection{Review of Wasserstein-$2$ distances}
 We first review the definition of optimal transport mapping functions in one-dimensional sample space. See \cite[Formula (6.0.3)] {AGS}.

For any two probability densities $p$, $q\in\mathcal{P}(\Omega)$ with finite second moments, the Wasserstein-$2$ distance is defined by: 
\begin{equation}\label{OT}
\textrm{W}_2(p,q):=\inf_{T}\quad\sqrt{\int_\Omega |T(x)-x|^2q(x)dx},
\end{equation}
where the infimum is taken over all continuous mapping function $T\colon \Omega\rightarrow\Omega$ that pushforwards $q$ to $p$. We also write the pushforward operation by $T_\#q=p$, which represents that the Monge-Amper{\'e} equation holds:  
\begin{equation}\label{MA}
p(T(x))\cdot T'(x)=q(x).
\end{equation}

The optimal mapping $T$ is a monotone function with closed form formulas. Denote the cumulative distribution functions (CDFs) $F_p$, $F_q$ of probability density function $p$, $q$, respectively, such that 
\begin{equation*}
F_p(x)=\int_{-\infty}^xp(y)dy, \quad F_q(x)=\int_{-\infty}^xq(y)dy. 
\end{equation*}
Denote the quantile functions of probability densities $p$, $q$ by 
\begin{equation*}
Q_p(u)=F_p^{-1}(u),\qquad Q_q(u)=F_q^{-1}(u).
\end{equation*}
where we denote $F_p^{-1}$, $F_q^{-1}$ are inverse CDFs of $p$, $q$, respectively. From the integration on both sides of equation \eqref{MA} with respect to $x$, then 
\begin{equation*}
F_p(T(x))=F_q(x).
\end{equation*} 
Thus, the optimal transport mapping function satisfies 
\begin{equation}\label{sMA}
T(x):=F_p^{-1}(F_q(x))=Q_p(F_q(x)). 
\end{equation}
Equivalently, the squared Wasserstein-$2$ distance satisfies  
\begin{equation*}
\begin{split}
\textrm{W}_2(p,q)^2=&{\int_\Omega |Q_p(F_q(x))-x|^2q(x)dx}\\
=&{\int_0^1|Q_p(u)-Q_q(u)|^2du},
\end{split}
\end{equation*}
 where we denote $u=F_q(x)$ with $u\in [0,1]$. 

There is a Kantorovich formulation and duality formula for the Wasserstein-$2$ distance \eqref{OT}. We can write the linear programming formulation of one-half squared Wasserstein-2 distance: 
\begin{equation*}
\frac{1}{2}\textrm{W}_2(p,q)^2=\inf_{\pi} \int_{\Omega}\int_\Omega\frac{1}{2}|x-y|^2\pi(x,y)dxdy,
\end{equation*}
where the infimum is taken among all joint distributions $\pi\in L^1(\Omega^2; \mathbb{R})$ with marginal densities $p$, $q$, respectively, such that 
\begin{equation*}
\int_\Omega \pi(x,y)dx=p(y), \quad \int_\Omega \pi(x,y)dy=q(x),\quad \pi(x,y)\geq 0.
\end{equation*}
The Kantorovich duality formula means that the Wasserstein-2 distance can be represented by: 
\begin{equation*}
\frac{1}{2}\textrm{W}_2(p,q)^2=\int_\Omega \Phi_1(y)p(y)dy-\int_\Omega \Phi_0(x)q(x)dx, 
\end{equation*}
where $\Phi_0$, $\Phi_1\in C(\Omega; \mathbb{R})$ are a pair of functions, Kantorovich duality variables, corresponding to densities $q$, $p$, respectively, such that 
\begin{equation*}
\Phi_1'(T(x))=\Phi'_0(x)=T(x)-x.
\end{equation*}
By taking the integration of the above formula, we have
\begin{equation}\label{Phi}
\Phi_0(x)=\int_{0}^xT(y)dy-\frac{|x|^2}{2}+c_0, \quad \Phi_1(x)=\frac{|x|^2}{2}-\int_{0}^x T^{-1}(y)dy+c_1.
\end{equation}
Here $c_0$, $c_1\in\mathbb{R}$ are constants, and $T^{-1}$ is the inverse function of the optimal mapping function. From equation \eqref{sMA}, $T^{-1}(x)=Q_q(F_p(x))$. 

\subsection{Transport $f$-divergences}
In this subsection, we define $f$-divergences in Wasserstein-$2$ space. 
\begin{definition}[Transport $f$-divergence] 
Given a positive convex function $f\colon\Omega\rightarrow\mathbb{R}_+$, such that $f(1)=0$. Define 
a functional $\mathrm{D}_{\rt,f}\colon \mathcal{P}(\Omega)\times\mathcal{P}(\Omega)\rightarrow\mathbb{R}_+$, such that
\begin{equation}\label{f-div-1}
\begin{split}
\mathrm{D}_{\rt, f}(p\|q)=&\int_\Omega f(\frac{q(x)}{p(T(x))}) q(x)dx,
\end{split}
\end{equation}
where $T$ is the monotone mapping function, such that $T_\#q=p$. We call $\mathrm{D}_{\rt,f}$ the {\em transport $f$-divergence}.
\end{definition}
We first present several equivalent formulations of transport $f$-divergences.
\begin{proposition}[Equivalent formulations]\label{prop1}
The following equivalent formulations hold: 
\begin{itemize}
\item[(i)]\begin{equation}\label{f-div-3}
\begin{split}
\mathrm{D}_{\rt,f}(p\|q)=&\int_\Omega f(T'(x))q(x)dx\\
=&\int_\Omega f(\frac{q(x)}{p(T(x))})q(x)dx\\
=&\int_\Omega f(\frac{q(T^{-1}(x))}{p(x)})p(x)dx.
\end{split}
\end{equation}
\item[(ii)]\begin{equation}\label{f-div}
\mathrm{D}_{\rt,f}(p\|q)=\int_0^1 f(\frac{Q'_p(u)}{Q'_q(u)})du.
\end{equation}
Here $Q'_p(u)=\frac{d}{du}Q_p(u)$, $Q'_q(u)=\frac{d}{du}Q_q(u)$ are quantile density functions of densities $p$, $q$, respectively. 
\item[(iii)] Let $p_{\mathrm{ref}}\in\mathcal{P}(\Omega)$ be a reference measure. Let $T_p$, $T_q\colon \Omega\rightarrow\Omega$ be the monotone mapping functions, which pushforward probability densities $p$, $q$ to $p_{\mathrm{ref}}$, respectively. I.e., $p=(T_{p})_\# p_{\mathrm{ref}}$ and $q=(T_{q})_\#p_{\mathrm{ref}}$. Then 
\begin{equation}\label{f-div-2}
\mathrm{D}_{\rt,f}(p\|q)=\int_\Omega f(\frac{T'_p(z)}{T'_q(z)}) p_{\mathrm{ref}}(z)dz.
\end{equation}
\end{itemize}
\end{proposition}
\begin{proof} 
\noindent(i) The optimal transport mapping function is monotone. From the Monge-Amper{\'e} equation \eqref{MA}, we have
\begin{equation*}
\mathrm{D}_{T,f}(p\|q)=\int_\Omega f( T'(x))q(x)dx=\int_\Omega f(\frac{q(x)}{p(T(x))})q(x)dx.
\end{equation*}
Similarly, denote $(T^{-1})_\#p=q$, then the following Monge-Amper{\'e} equation holds:
\begin{equation*}
q(T^{-1}(x))\frac{d}{dx}T^{-1}(x)=p(x).
\end{equation*}
Let $x=T^{-1}(\tilde x)$, then 
\begin{equation*}
\begin{split}
\mathrm{D}_{T,f}(p\|q)=&\int_\Omega f(\frac{q(x)}{p(T(x))})q(x)dx\\
=&\int_\Omega f(\frac{q(T^{-1}(\tilde x))}{p(\tilde x)})q(T^{-1}(\tilde x))dT^{-1}(\tilde x)\\
=&\int_\Omega f(\frac{q(T^{-1}(\tilde x))}{p(\tilde x)})p(\tilde x)dx. 
\end{split}
\end{equation*}
Hence we derive the result.

\noindent(ii) Denote the change of variable formula by
\begin{equation*}
u=F_q(x)=\int_{-\infty}^xq(y)dy.
\end{equation*}
Hence
\begin{equation*}
\frac{du}{dx}=q(x), \qquad x=Q_q(u).
\end{equation*}
From the chain rule, we have 
\begin{equation*}
q(x)=\frac{du}{dx}=({\frac{dx}{du}})^{-1}=\frac{1}{Q'_q(u)}.
\end{equation*}
Thus 
\begin{equation*}
\begin{split}
\frac{Q'_p(u)}{Q'_q(u)}=&\frac{d}{du}Q_p(F_q(x))\cdot\frac{du}{dx}\\
=&\frac{d}{dx}Q_p(F_q(x))\\
=& T'(x),
\end{split}
\end{equation*}
where the second equality holds by the chain rule, and the last equality holds because $T(x)=Q_p(F_q(x))$.
Substituting the above calculations into \eqref{f-div-1}, we derive \eqref{f-div}.

\noindent(iii) Denote the change of variable formula by
\begin{equation*}
x=T_q(z).
\end{equation*}
Since $(T_{q})_\#p_{\mathrm{ref}}=q$, then 
\begin{equation*}
p_{\mathrm{ref}}(z)=q(T_q(z))T'_q(z).
\end{equation*}
In other words, 
\begin{equation*}
p_{\mathrm{ref}}(z)dz=q(x)dx.
\end{equation*}
Note that 
\begin{equation*}
T_p(z)=Q_p(F_{p_\mathrm{ref}}(z)),\quad T_q(z)=Q_q(F_{p_\mathrm{ref}}(z)).
\end{equation*}
Hence 
\begin{equation*}
T(x)=T(T_q(z))=T_p(z), 
\end{equation*}
and 
\begin{equation*}
\begin{split}
 T'(x)=&\frac{d}{dz} T(T_q(z))\cdot \frac{dz}{dx}\\
=&\frac{d}{dz} T_p(z)\cdot (\frac{dx}{dz})^{-1}\\
=&\frac{T'_p(z)}{T'_q(z)}.
\end{split}
\end{equation*}
Substituting the above calculations into \eqref{f-div-1}, we derive formulation \eqref{f-div-2}. 
\end{proof}
\subsection{Properties}
We next present several properties of transport $f$-divergences.  
\begin{proposition}[Properties]
The following properties hold:
\begin{itemize}
\item[(i)] Nonnegativity: The transport $f$-divergence is nonnegative:
\begin{equation*}
\mathrm{D}_{\rt,f}(p\|q)\geq 0.
\end{equation*}
And $\mathrm{D}_{\rt,f}(p\|q)=0$ if and only if $p$ equals to $q$ up to a constant shrift in their variables. I.e., there exists a constant $c\in\mathbb{R}$, such that 
\begin{equation*}
p(x+c)=q(x).
\end{equation*}
\item[(ii)] Entropy: Let $\Omega=[0,1]$ and $q(x)=1$. Then the transport $f$-divergence equals to the following negative entropy. 
\begin{equation*}
\mathrm{D}_{\rt,f}(p\|1)=\int_\Omega f(\frac{1}{p(x)})p(x)dx.
\end{equation*}
\item[(iii)] Additivity and scaling: Suppose $f_1$, $f_2$ are convex functions and $a>0$, then 
\begin{equation*}
\mathrm{D}_{\rt,f_1+af_2}(p\|q)=\mathrm{D}_{\rt,f_1}(p\|q)+a\mathrm{D}_{\rt,f_2}(p\|q).
\end{equation*}
\item[(iv)] Duality:
\begin{equation*}
\mathrm{D}_{\rt,f}(p\|q)=\mathrm{D}_{\rt,\hat f}(q\|p),
\end{equation*}
where $\hat f(u)=f(\frac{1}{u})$. 
\item[(v)] Transport invariance: Let $k\colon\Omega\rightarrow\Omega$ be a smooth inverse mapping function. Denote $k^{-1}$ be the inverse function of $k$. 
Let 
\begin{equation*}
\tilde q=(k^{-1})_\#q,
\end{equation*}
and 
\begin{equation*}
\tilde p=\tilde T_\# \tilde q,
\end{equation*}
with 
\begin{equation}\label{tt}
\tilde T(x)=\int_{0}^xT'(k(y))dy+c,  
\end{equation}
for any constant $c\in\mathbb{R}$.
Then 
\begin{equation*}
\mathrm{D}_{\rt,f}(p\|q)=\mathrm{D}_{\rt,f}(\tilde p\|\tilde q).
\end{equation*}
\item[(vi)] Transport convexity in the first variable: Given any densities $p_1, p_2, q\in\mathcal{P}(\Omega)$ with $p_1={T_1}\# q$ and $p_2={T_2}_\# q$, where $T_1$, $T_2$ are two monotone mapping functions, respectively. Denote 
\begin{equation*}
p_\lambda=\Big(\lambda T_1+(1-\lambda)T_{2}\Big)_\# q. 
\end{equation*}
Then for any constant $\lambda\in [0,1]$, we have
\begin{equation*}
\mathrm{D}_{\rt,f}(p_\lambda\|q)\leq \lambda \mathrm{D}_{\rt,f}(p_1\|q)+(1-\lambda)\mathrm{D}_{\rt,f}(p_2\|q).
\end{equation*}
\end{itemize}
\end{proposition}
\begin{proof}
The proofs are based on the definitions of transport $f$-divergences. 

\noindent(i) The nonnegativity follows from the fact that $f$ is a positive convex function. Suppose $\mathrm{D}_{\rt,f}(p\|q)=0$.  Following \eqref{f-div-1}, we have 
$f( T'(x))=0$, for $x$ in the support of $q(x)$. In other words, $T'(x)=1$, i.e. $T(x)=x+c$, where $c\in\mathbb{R}$ is a constant value. From $p=T_\# q$, we derive the result.   

\noindent(ii) Denote $T_\# q=p$ and $q(x)=1$, then  
\begin{equation*}
\begin{split}
\mathrm{D}_{\rt,f}(p\|q)=&\int_\Omega f(\frac{1}{p(T(x))})dx\\
=&\int_\Omega f(\frac{1}{p(T(x))})p(T(x))T'(x)dx\\
=&\int_\Omega f(\frac{1}{p(T(x))})p(T(x))dT(x).
\end{split}
\end{equation*}
Denote $y=T(x)$, then $\mathrm{D}_{\rt,f}(p\|q)=\int_\Omega f(\frac{1}{p(y)})p(y)dy$.

\noindent(iii) The additivity and scaling property holds from the definition:  
\begin{equation*}
\mathrm{D}_{\rt, f_1+af_2}(p\|q)=\int_\Omega \Big(f_1( T'(x))+af_2( T'(x))\Big) q(x)dx=\mathrm{D}_{\rt, f_1}(p\|q)+a\mathrm{D}_{\rt,f_2}(p\|q).
\end{equation*}
(iv) The proof follows from the formulation of transport $f$-divergence defined in \eqref{f-div-2}. Note
\begin{equation*}
\begin{split}	
\mathrm{D}_{\rt,f}(p\|q)=&\int_\Omega f(\frac{T'_p(z)}{T'_q(z)}) p_{\mathrm{ref}}(z)dz\\
%%=&\int_\Omega f\Big(\frac{1}{\frac{T'_q(z)}{T'_p(z)}}\Big) p_{\mathrm{ref}}(z)dz\\
=&\int_\Omega \hat f\Big(\frac{T'_q(z)}{T'_p(z)}\Big) p_{\mathrm{ref}}(z)dz.
\end{split}
\end{equation*}
\noindent(v) The proof follows from the definition of pushforward operator. From $\tilde q=(k^{-1})_\#q$ and $q=k_\#\tilde q$, we obtain 
\begin{equation*}
q(k(y))k'(y)=\tilde q(y).
\end{equation*}
From the definition of $\tilde T$ in \eqref{tt}, we have 
\begin{equation*}
 \frac{d}{dx}\tilde T(x)=\frac{d}{dx}\int_{0}^xT'(k(y))dy=T'(k(x)).
\end{equation*}
Following the above two equalities, we have
\begin{equation*}
\begin{split}
\mathrm{D}_{\rt,f}(\tilde p\|\tilde q)=&\int_\Omega f(T'(k(y))) \tilde q(y)dy\\
=&\int_\Omega f(T'(k(y))) q(k(y))k'(y)dy\\
=&\int_\Omega f(T'(k(y)))q(k(y))dk(y)\\
=&\int_\Omega f( T'(x))q(x)dx\\
=&\mathrm{D}_{\rt,f}(p\|q),
\end{split}
\end{equation*}
where we let $x=k(y)$ in the fourth equality.

\noindent(vi) The proof follows directly from the convexity of function $f$. Notice that
\begin{equation*}
\begin{split}
\mathrm{D}_{f, \rt}(p_\lambda\|q)=&\int_\Omega f(\lambda T'_{1}(x)+(1-\lambda)T'_{2}(x))q(x)dx\\
\leq &  \int_\Omega \Big(\lambda f(T'_{1}(x))+(1-\lambda)f(T'_{2}(x))\Big)q(x)dx\\
=&\lambda \mathrm{D}_{\rt,f}(p_1\|q)+(1-\lambda)\mathrm{D}_{\rt,f}(p_2\|q).
\end{split}
\end{equation*}
\end{proof}
\subsection{Variational formulations}
In this subsection, we present variational formulations for transport $f$-divergences. 
\begin{theorem}[Transport $f$-dualities]
Denote
\begin{equation*}
\hat f(u)=f(\frac{1}{u}). 
\end{equation*}
 Denote $\hat f^*$ as the conjugate function of $\hat f$, with $\hat f^*(v)=\sup_{u\in\mathbb{R}}\big\{uv-\hat f(u)\big\}$.
Assume $\hat f$ is strictly convex with respect to the variable $u$. Then the transport $f$-divergence satisfies
\begin{equation}\label{Psi}
\mathrm{D}_{\rt, f}(p\|q)=\sup_{\Psi}~\int_\Omega \Psi(x)p(T(x))dx-\int_\Omega \hat f^*(\Psi(x))q(x)dx,
\end{equation}
where the supreme is over all continuous functions $\Psi\in C(\Omega; \mathbb{R})$, and $T$ is a monotone mapping function, such that $T_\#q=p$. The optimality condition of supreme problem \eqref{Psi} is described below. Denote the transport $f$-duality variable:
\begin{equation}\label{Psi_opt}
\Psi_{\mathrm{opt}}(x):=\hat f'(T'(x)),
\end{equation} 
Then 
\begin{equation*}
\mathrm{D}_{\rt, f}(p\|q)=\int_\Omega \Psi_{\mathrm{opt}}(x)p(T(x))dx-\int_\Omega \hat f^*(\Psi_{\mathrm{opt}}(x))q(x)dx. 
\end{equation*}
\end{theorem}
\begin{proof}
The proof follows from the convex conjugate of a function $\hat f$. Notice that $\hat f(u)=\sup_{v\in \mathbb{R}}\big\{uv-\hat f^*(v)\big\}$, where $u=(\hat f^*)'$. Thus, 
\begin{equation*}
\begin{split}
\mathrm{D}_{\rt, f}(p\|q)=&\int_\Omega f( T'(x))q(x)dx=\int_\Omega \hat f(\frac{1}{ T'(x)})q(x)dx\\
=&\sup_{\Psi}~~\int_\Omega \Big(\Psi(x)\cdot\frac{1}{ T'(x)}-\hat f^*(\Psi(x))\Big)q(x)dx\\
=&\sup_{\Psi}~~\int_\Omega \Psi(x)\cdot\frac{q(x)}{ T'(x)}dx-\int_\Omega \hat f^*(\Psi(x))q(x)dx\\
=&\sup_{\Psi}~~\int_\Omega \Psi(x) p(T(x))dx-\int_\Omega \hat f^*(\Psi(x))q(x)dx,
\end{split}
\end{equation*}
where the last equality holds from equation \eqref{MA}. And the optimality condition implies that 
\begin{equation*}
\frac{1}{T'(x)}=(\hat f^*)'(\Psi(x)). 
\end{equation*}
From the assumption, $\hat f$ is a strictly convex function, then $(\hat f^*)'$ is invertible. Then, we have the optimality condition $\Psi_{\mathrm{opt}}(x)=((\hat f^*)')^{-1}(\frac{1}{T'(x)})=\hat f'(\frac{1}{T'(x)})$. Substituting $\Psi_{\mathrm{opt}}$ into formulation \eqref{Psi}, we finish the proof. 
\end{proof}

We also present a relation between the transport $f$-dualities and the Kantorvich dualities in Wasserstein-2 distances. \begin{theorem}[Transport $f$-Kantorvich dualities]
The following equation holds:
\begin{equation*}
\Psi_{\mathrm{opt}}(x)=\hat f'(\frac{1}{\Phi''_0(x)+1}), 
\end{equation*}
where $\Phi_0(x)$ is the Kantorovich duality variable defined in \eqref{Phi}, and $\Psi_{\mathrm{opt}}$ is the transport $f$-duality variable defined in \eqref{Psi_opt}. In addition, the transport $f$-divergence is reformulated below:
\begin{equation*}
\mathrm{D}_{\rt, f}(p\|q)=\int_\Omega \hat f'(\frac{1}{\Phi''_0(x)+1})p(\Phi'_0(x)+x)dx-\int_\Omega \hat f^*(\frac{1}{\Phi''_0(x)+1})q(x)dx. 
\end{equation*}
\end{theorem}
\begin{proof}
The proof follows from a direct calculation. From Kantorovich duality, we have $\Phi'_0(x)=T(x)-x$. By taking the derivative in above formula, we have
\begin{equation*}
\Phi''_0(x)=T'(x)-1. 
\end{equation*}
Substituting the above formula into \eqref{Psi_opt} and \eqref{Psi}, we finish the proof.
\end{proof}
\subsection{Local behaviors and Taylor expansions}
We last formulate the local behaviors and Taylor expansions of transport $f$-divergences. 
\begin{theorem}[Local behaviors]\label{thm2}
Assume $f\in C^2$ and $f(1)=f'(1)=0$, then 
\begin{equation*}
\lim_{\lambda\rightarrow 0}\frac{1}{\lambda^2}\mathrm{D}_{\mathrm{T},f}(p_\lambda\|q)=\frac{f''(1)}{2}\int_\Omega | T'(x)-1|^2q(x)dx,
\end{equation*}
where $p_\lambda\in\mathcal{P}(\Omega)$, $\lambda\in [0,1]$, is the geodesic in Wasserstein-$2$ space connecting probability densities $q$, $p$. In other words, 
\begin{equation*}
p_\lambda=\Big((1-\lambda)\mathrm{id}+\lambda T\Big)_\#q,
\end{equation*}
where $\mathrm{id}(x):=x$ is an identity mapping function. 
\end{theorem}
\begin{proof}
By the definition of $p_\lambda$, we have 
\begin{equation*}
\begin{split}
\mathrm{D}_{\mathrm{T}, f}(p_\lambda\|q)=&\int_\Omega f\big((1-\lambda)+\lambda T'(x)\big)q(x)dx\\
=&\int_\Omega f\big(1+\lambda( T'(x)-1)\big)q(x)dx.
\end{split}
\end{equation*}
By the Taylor expansion of function $f\big((1-\lambda)+\lambda T'(x)\big)$, we have 
\begin{equation*}
\begin{split}
\mathrm{D}_{\mathrm{T}, f}(p_\lambda\|q)=&\int_\Omega\Big(f(1)+\lambda f'(1)( T'(x)-1)+\frac{\lambda^2f''(1)}{2}| T'(x)-1|^2\Big) q(x)dx+o(\lambda^2)\\
=& \frac{\lambda^2f''(1)}{2}\int_\Omega| T'(x)-1|^2q(x)dx+o(\lambda^2).
\end{split}
\end{equation*}
In above derivations, we apply the fact that $f(1)=f'(1)=0$ in the second equality. This finishes the proof. 
\end{proof}

\begin{theorem}[Taylor expansions in Wasserstein-$2$ space]
Assume $f\in C^4$ and $f(1)=f'(1)=0$. Then the following equation holds:
\begin{equation*}
\begin{split}
\mathrm{D}_{\rt,f}(p\|q)=&\int_0^1 \Big[\frac{f''(1)}{2}|\frac{Q'_p(u)-Q'_q(u)}{Q'_q(u)}|^2+\frac{f'''(1)}{6} \Big(\frac{Q'_p(u)-Q'_q(u)}{Q'_q(u)}\Big)^3\Big]du\\
&+O(\int_0^1 |\frac{Q'_p(u)-Q'_q(u)}{Q'_q(u)}|^4 du). 
\end{split}
\end{equation*}
We also represent the above formula in terms of Kantorovich duality variable $\Phi_0$ defined in \eqref{Phi}.  
\begin{equation*}
\begin{split}
\mathrm{D}_{\rt,\alpha}(p\|q)=&\int_\Omega\Big[\frac{f''(1)}{2}|\Phi_0''(x)|^2+\frac{f'''(1)}{6}(\Phi_0''(x))^3\Big]q(x)dx\\
&+O(\int_\Omega |\Phi_0''(x)|^4q(x)dx). 
\end{split}
\end{equation*}
\end{theorem}
\begin{proof}
The proof is based on a direct calculation. Firstly, from quantile density function formulation \eqref{f-div}, 
we have 
\begin{equation*}
\mathrm{D}_{\rt,f}(p\|q)=\int_0^1 f(1+h(u))du,
\end{equation*}
where $h(u):=\frac{Q'_p(u)-Q'_q(u)}{Q'_q(u)}$. From the Taylor expansion of function $f$ at $1$ and $f'(1)=f''(1)=0$, we have
\begin{equation*}
f(1+h(u))=\frac{1}{2}f''(1)h(u)^2+\frac{1}{6}f'''(1)h(u)^3+O(|h(u)|^4).
\end{equation*}
Secondly, we note that the Kantorovich duality condition \eqref{Phi} holds. Denote $\Phi'_0(x)=T(x)-x=Q_p(F_q(x))-x$, and 
\begin{equation*}
\Phi_0''(x)=\frac{d}{dx}Q_p(F_q(x))-1=\frac{Q'_p(F_q(x))}{\frac{1}{q(x)}}-1. 
\end{equation*}
Denote $u=F_q(x)$. Then we have
\begin{equation*}
\int_\Omega (\Phi''(x))^kp(x)dx=\int_0^1 (\frac{Q'_p(u)}{Q'_q(u)}-1)^k du, 
\end{equation*}
for $k=2,3$. This finishes the proof. 
\end{proof}

\section{Examples}\label{sec4}
In this section, we list several examples of transport $f$-divergences.  
\begin{example}[Transport total variation]
Let $f(u)=|u-1|$, then 
\begin{equation*}
\begin{split}
\mathrm{D}_{\rt,f}(p\|q)=\mathrm{D}_{\mathrm{TTV}}(p\|q)=&\int_\Omega |T'(x)-1|q(x)dx\\
=&\int_0^1|\frac{Q'_p(u)}{Q'_q(u)}-1|du.  
\end{split}
\end{equation*}
We call $\mathrm{D}_{\mathrm{TTV}}$ the transport total variation (TTV). Unfortunately, it is not a distance, since $\mathrm{D}_{\mathrm{TTV}}(p\|q)\neq \mathrm{D}_{\mathrm{TTV}}(q\|p)$.
\end{example}
\begin{example}
Let $f(u)=|u-1|^2$, then 
\begin{equation*}
\begin{split}
\mathrm{D}_{\rt,f}(p\|q)=&\int_\Omega |T'(x)-1|^2q(x)dx\\
=&\int_0^1|\frac{Q'_p(u)}{Q'_q(u)}-1|^2du.
\end{split}
\end{equation*}
\end{example}
\begin{example}[Squared transport Hessian distance \cite{LiHess1,LiHess2}]
Let $f(u)=|\log u|^2$, then 
\begin{equation*}
\begin{split}
\mathrm{D}_{\rt,f}(p\|q)=\mathrm{Dist}_{\mathrm{TH}}(p,q)^2=&\int_\Omega |\log T'(x)|^2q(x)dx\\
=&\int_0^1|\log \frac{Q'_p(u)}{Q'_q(u)}|^2du.
\end{split}
\end{equation*}
We call $\mathrm{Dist}_{\mathrm{TH}}(p\|q)$ the transport Hessian distance; see derivations in \cite{LiHess2}. \end{example}
\begin{example}[Transport KL divergence \cite{LiD}]
Let $f(u)=u-\log u-1$, then 
\begin{equation*}
\begin{split}
\mathrm{D}_{\rt,f}(p\|q)=\mathrm{D}_{\mathrm{TKL}}(p\|q)=&\int_\Omega \Big( T'(x)-\log T'(x)-1\Big)q(x)dx\\
=&\int_0^1\Big(\frac{Q'_p(u)}{Q'_q(u)}-\log\frac{Q'_p(u)}{Q'_q(u)}-1\Big)du. 
\end{split}
\end{equation*}
We call $\mathrm{D}_{\mathrm{TKL}}$ the transport KL divergence (TKL). It is the Bregman divergence of the negative Boltzmann-Shannon entropy in Wasserstein-$2$ space.
\end{example}

\begin{example}[Transport Jenson-Shannon divergence \cite{LiD}]
Let $f(u)=-\frac{1}{2}\log\frac{u}{\frac{1}{4}|u+1|^2}$, then 
\begin{equation*}
\begin{split}
\mathrm{D}_{\rt,f}(p\|q)=\mathrm{D}_{\mathrm{TJS}}(p\|q)=&-\frac{1}{2}\int_\Omega \log\frac{ T'(x)}{\frac{1}{4}| T'(x)+1|^2} q(x)dx\\
=& -\frac{1}{2}\int_0^1\log\frac{Q'_p(u)\cdot Q'_q(u)}{\frac{1}{4}|{Q'_p(u)}+{Q'_q(u)}|^2}du.
\end{split}
\end{equation*}
We name it the transport Jenson-Shannon divergence (TJS). In fact, the TJS is a symmetrized transport KL divergence, i.e. 
\begin{equation*}
\mathrm{D}_{\mathrm{TJS}}(p\|q)=\frac{1}{2}\Big(\mathrm{D}_{\mathrm{TKL}}(p\|p_{\frac{1}{2}})+\mathrm{D}_{\mathrm{TKL}}(q\|p_{\frac{1}{2}})\Big),
\end{equation*}
where $p_{\frac{1}{2}}=(\frac{1}{2}\mathrm{id}+\frac{1}{2}T)_\#q$ is the geodesic midpoint (Barycenter) between densities $p$, $q$ in Wasserstein-$2$ space. 
\end{example}

\begin{example}[Transport $\alpha$-divergences \cite{LiAlpha}]
Let 
\begin{equation*}
f_\alpha(u)=\left\{\begin{aligned}&\frac{1}{\alpha^2}(u^{\alpha}-\alpha\log u-1), &\quad \alpha\neq 0;\\
&\frac{1}{2}|\log u|^2,&\quad \alpha=0.
\end{aligned}\right.
\end{equation*}
Then 
\begin{equation*}
\begin{aligned}
\mathrm{D}_{\rt,f}(p\|q)=\mathrm{D}_{\rt,\alpha}(p\|q)=&\frac{1}{\alpha^2}\int_\Omega \Big(( T'(x))^{\alpha}-\alpha\log T'(x)-1\Big)q(x)dx\\
=&\frac{1}{\alpha^2}\int_0^1\Big(\big(\frac{ Q'_p(u)}{Q'_q(u)}\big)^{\alpha}-{\alpha}\log\frac{Q'_p(u)}{Q'_q(u)}-1\Big)du.
\end{aligned}
\end{equation*}
We call $\mathrm{D}_{\rt, \alpha}$ the transport $\alpha$-divergence. The divergence $\mathrm{D}_{\rt, \alpha}$ is a generalization of transport KL divergence. E.g. if $\alpha=1$, we have $\mathrm{D}_{\rt,1}(p\|q)=\mathrm{D}_{\mathrm{TKL}}(p\|q)$. If $\alpha=0$, we obtain $\mathrm{D}_{\rt, 0}(p\|q)=\frac{1}{2}\mathrm{D}_{\mathrm{TH}}(p,q)^2$. Transport $\alpha$-divergences are related with transport Hessian metric structures \cite{LiAlpha}, which are analogs of information geometry methods in Wasserstein-2 spaces; see \cite{AC, IG2, CA}. 
%%We leave these detailed derivations in a sequential work. 
\end{example}
\section{Examples}\label{sec5}
In this section, we present two analytical examples of transport $f$-divergences in either location scale families, or generative models. 
\begin{example}[Location scale family]
Let $p_X$, $p_Y$ be two one-dimensional local scale probability densities. Suppose
\begin{equation*}
X\sim p_X,\quad Y:=T(X)=\mu_Y+\frac{\sigma_Y}{\sigma_X}(X-\mu_X)\sim p_Y,
\end{equation*}
where $\mu_X, \mu_Y\in\mathbb{R}$, $\sigma_X, \sigma_Y>0$ are mean values and standard derivations of random variables $X$, $Y$, respectively. In this case, $ T'(x)=\frac{\sigma_Y}{\sigma_X}$. Hence the transport $f$-divergence forms  
\begin{equation*}
\mathrm{D}_{\rt,f}(p_X\|p_Y)=f(\frac{\sigma_Y}{\sigma_X}).
\end{equation*}
We note that the transport $f$-divergence does not depend on the location variable of probability density functions. 
\end{example}
\begin{example}[Generative family]\label{exp13}
Suppose $\Omega=\mathbb{R}^1$. Let $p_X$, $p_Y$ be constructed from generative models. Consider a latent random variable $Z\sim p_{\mathrm{ref}}\in\mathcal{P}(\Omega)$, where $p_{\mathrm{ref}}$ is a given smooth reference measure. Denote a smooth invertible map function $G\in C^1(\Omega\times \Theta; \Omega)$, where $\Theta\subset\mathbb{R}^n$, $n\in N$, is a parameter space. Let $\theta_X$, $\theta_Y\in\Theta$ and consider 
\begin{equation*}
X=G(Z, \theta_X)\sim p_X,\quad Y=G(Z, \theta_Y)\sim p_Y. 
\end{equation*}
By Proposition \ref{prop1} (iii), the transport $f$-divergences satisfies 
\begin{equation*}
\mathrm{D}_{\rt,f}(p_X\|p_Y)=\mathbb{E}_{Z\sim p_{\mathrm{ref}}}\Big[f(\frac{\partial_ZG(Z, \theta_X)}{\partial_ZG(Z, \theta_Y)})\Big].
\end{equation*}
We remark that transport $f$-divergences depend on the derivative of generative mapping functions with respect to the input variable $Z$. 
\end{example}
\section{Discussion}
In this paper, we propose a class of transport $f$-divergences. The proposed transport type divergence functionals are built from the derivative of pushforward mapping functions. These divergence functionals have convexity properties in terms of mapping functions, which contrasts with $f$-divergences. 

The study of transport $f$-divergences in multi-dimensional sample space are left in future works. In general, transport $f$-divergences can be defined from ``matrix divergences'', where the matrix refers to the Jacobian matrix of pushforward mapping function; see \cite{LiD}. We shall also investigate the convexity properties, inequalities, and variational algorithms for transport $f$-divergences towards generative AI-related sampling problems and Bayesian inverse problems.

\noindent\textbf{Acknowledgements}. {W. Li's work is supported by AFOSR YIP award No. FA9550-23-1-0087, NSF RTG: 2038080, and NSF DMS: 2245097.}

\end{document}